\documentclass[12pt]{amsart}
\usepackage{pinlabel}
\usepackage{amsfonts,amsmath,amscd,xcolor,graphicx,amssymb}
\usepackage{caption}
\usepackage{epstopdf}
\usepackage{a4wide}

\theoremstyle{plain}
\newtheorem{theorem}{Theorem}[section]

\newtheorem{lemma}[theorem]{Lemma}

\newtheorem*{Powell Conjecture}{Powell Conjecture}

\theoremstyle{definition}

\newcommand{\N}{\operatorname{N}}

\theoremstyle{remark}

\newtheorem{claim}{Claim}

\makeatletter
\@namedef{subjclassname@2020}{\textup{2020} Mathematics Subject Classification}
\makeatother

\begin{document}

\title[Reducing spheres for the genus-4 Heegaard surface in the 3-sphere]{A note on reducing spheres for the genus-4 Heegaard surface in the 3-sphere}

\author[S. Cho]{Sangbum Cho}
\thanks{The first-named author is supported by the National Research Foundation of Korea(NRF) grant funded by the Korea government(MSIT) (RS-2024-00456645).}
\address{Department of Mathematics Education, Hanyang University, Seoul 04763, Korea}
\email{scho@hanyang.ac.kr}

\author[Y. Koda]{Yuya Koda}
\thanks{The second-named author is supported by JSPS KAKENHI Grant Numbers JP23K20791, JP24K06744, and JP25K24558.}
\address{Department of Mathematics, Hiyoshi Campus, Keio University, Yokohama 223-8521, Japan, and
International Institute for Sustainability with Knotted Chiral Meta Matter (WPI-SKCM$^2$), Hiroshima University, 1-3-1 Kagamiyama, Higashi-Hiroshima, Hiroshima 739-8531, Japan}
\email{koda@keio.jp}

\author[J. H. Lee]{Jung Hoon Lee}
\thanks{The third-named author is supported by the National Research Foundation of Korea(NRF) grant funded by the Korea government(MSIT) (RS-2026-25469563).}
\address{Department of Mathematics and Institute of Pure and Applied Mathematics, Jeonbuk National University, Jeonju 54896, Korea}
\email{junghoon@jbnu.ac.kr}


\begin{abstract}
For the genus-$4$ Heegaard surface in the $3$-sphere, we present a sufficient condition for a non-separating weak reducing pair to be separated by a reducing sphere for the surface. As a consequence, we reduce the connectivity problem in the reducing sphere complex for the surface to the problem of showing that any two vertices, whose representative reducing spheres are disjoint from a fixed non-separating compressing disk for the surface, are connected in the complex.
\end{abstract}

\maketitle

\section{Introduction}\label{sec1}

A {\em Heegaard splitting}, denoted by $(V, W; \Sigma)$, of a closed orientable $3$-manifold $M$ is a decomposition of $M$ into two handlebodies $V$ and $W$ of the same genus, where $\Sigma = \partial V = \partial W$ is called a {\em Heegaard surface}.
If the genus of $\Sigma$ is $g$, then $(V, W; \Sigma)$ is called a {\em genus-$g$ Heegaard splitting}.
In this paper, we consider Heegaard splittings of the $3$-sphere $S^3$.
It is well known from Waldhausen \cite{Wal68} that $S^3$ admits a genus-$g$ Heegaard splitting unique up to isotopy for each $g \ge 0$.
 
Let $(V, W; \Sigma)$ be a genus-$g$ Heegaard splitting of $S^3$.
A {\em reducing sphere} for $\Sigma$ is a $2$-sphere $P$ in $S^3$ that meets $\Sigma$ in a single essential circle. 
We have a reducing sphere for $\Sigma$ only when the genus $g$ is at least $2$, and there exist infinitely many reducing spheres for $\Sigma$ for each $g \geq 2$.
A {\em compressing disk} for $\Sigma$ is an essential disk in $V$ or in $W$.
A pair of compressing disks $D$ and $E$ for $\Sigma$ is called a {\em weak reducing pair} for $\Sigma$ if $D \subset V$ and $E \subset W$, and $\partial D$ is disjoint from $\partial E$.
We denote the pair simply by $D - E$ or by $E - D$, and call each of $D$ and $E$ a {\em weak reducing disk}. 
Again, we have a weak reducing pair for $\Sigma$ only when $g \geq 2$, and there exist infinitely many such pairs for each $g \geq 2$.
When both of the disks $D$ and $E$ are non-separating, we call the pair $D - E$ simply a {\em non-separating weak reducing pair}.

We say that a reducing sphere $P$ {\em separates} a weak reducing pair $D - E$ if $P$ is disjoint from $D \cup E$ and if $D$ and $E$ are contained in different components cut off by $P$.
Given any reducing sphere $P$, one can find easily a non-separating weak reducing pair separated by $P$ by a suitable choice of compressing disks for $\Sigma$ in different components cut off by $P$.
But, as shown in \cite{CKL25}, the converse holds only when $g \leq 3$. 
That is, given any non-separating weak reducing pair, there always exists a reducing sphere that separates the pair only when $g \leq 3$, and if $g \geq 4$, there exists a non-separating weak reducing pair (in fact, infinitely many such pairs) that admits no reducing sphere separating the pair.

Our first main result is to provide a reasonable sufficient condition for a non-separating weak reducing pair to be separated by a reducing sphere when $g = 4$, stated as follows.

\begin{lemma}\label{lem1.1} 
Let $(V, W; \Sigma)$ be the genus-$4$ Heegaard splitting of $S^3$, and let $D - E$ be a non-separating weak reducing pair for $\Sigma$.
If the surface obtained by compressing $\Sigma$ along $D \cup E$ is still compressible to both sides, then there exists a reducing sphere $P$ that separates the pair $D - E$.
\end{lemma}

We remark that, in the setting of Lemma~\ref{lem1.1}, the surface obtained by compressing $ \Sigma $ along $ D \cup E $ is connected, since both $ D $ and $ E $ are non-separating. 
This will be verified at the beginning of the proof of Lemma~\ref{lem1.1} in Section~\ref{sec6}.

In the authors' previous work \cite{CKL24}, the existence of a reducing sphere separating a given non-separating weak reducing pair plays a key role to show the connectivity of the reducing sphere complex for the genus-$3$ Heegaard surface $\Sigma$ in $S^3$.
In general, when $g \geq 3$, the {\em reducing sphere complex}, denoted by $\mathcal{R}(\Sigma)$, for the genus-$g$ Heegaard surface $\Sigma$ in $S^3$ is defined as follows.
The vertices of $\mathcal{R}(\Sigma)$ are the isotopy classes of reducing spheres for $\Sigma$, and a collection $\{v_0, v_1, \ldots, v_k\}$ of $k+1$ distinct vertices spans a $k$-simplex of $\mathcal{R}(\Sigma)$ if there exist representative spheres $P_0, P_1, \ldots, P_k$ of $v_0, v_1, \ldots, v_k$, respectively, satisfying that $|P_i \cap P_j \cap \Sigma_g| = 0$, for each $i, j \in \{0, 1, \ldots, k\}$ with $i \neq j$.

It is still an open question whether $\mathcal{R}(\Sigma)$ is connected or not if $g \geq 4$.
But, in the case of $g = 4$, Lemma \ref{lem1.1} enables us to reduce the problem of the connectivity of $\mathcal{R}(\Sigma)$ to the problem of showing that any two vertices, whose representative reducing spheres are disjoint from a fixed non-separating compressing disk for the surface, are connected in the complex.
The following is our second main result.

\begin{theorem}\label{thm1.2}
Let $(V, W; \Sigma)$ be a genus-$4$ Heegaard splitting of $S^3$.
Then the reducing sphere complex $\mathcal{R}(\Sigma)$ for $\Sigma$ is connected if and only if the following holds.
\begin{itemize}
\item Let $E$ be any non-separating compressing disk for $\Sigma$.
If $P$ and $Q$ are reducing spheres for $\Sigma$ disjoint from $E$, then the vertices of $\mathcal{R}(\Sigma)$ represented by $P$ and $Q$ belong to the same component of $\mathcal{R}(\Sigma)$.
\end{itemize}
\end{theorem}

The connectivity of the reducing sphere complex is closely related to a well known open question, called the Powell conjecture. 
The conjecture states that the genus-$g$ Goeritz group of $S^3$ is generated by four specific elements for each $g \geq 3$.
Here, the genus-$g$ {\em Goeritz group} of $S^3$, denoted by $\mathcal{G}_g$, is the group of isotopy classes of orientation-preserving self-homeomorphisms of $(S^3, \Sigma_g)$ that preserve $V$ and $W$ setwise, where $\Sigma_g$ is the genus-$g$ Heegaard surface that splits $S^3$ into two handlebodies $V$ and $W$.

For each $g \geq 3$, Zupan \cite{Zupan} showed that the Powell conjecture for $\mathcal{G}_k$ holds for all $3 \leq k \leq g$ if and only if the reducing sphere complex $\mathcal{R}(\Sigma_k)$ for $\Sigma_k$ is connected for all $3 \leq k \leq g$.
In \cite{CKL24}, it was shown that $\mathcal{R}(\Sigma_3)$ is connected, and hence the conjecture for $\mathcal{G}_3$ holds.
But, as stated above, we do not know whether the complex $\mathcal{R}(\Sigma_g)$ is connected or not for each $g \geq 4$, and so the conjecture is still open for $\mathcal{G}_g$ for each $g \geq 4$.
For a brief history of the Goeritz groups of $S^3$ and the Powell conjecture, we refer the reader to the introduction of \cite{CKL24}.

\medskip

Throughout the paper, all $3$-manifolds are assumed to be orientable.
For a subspace $X$ of a topological space, we let $\N(X)$, $\partial X$, and $\overline{X}$ denote a regular neighborhood, the boundary, and the closure of $X$, respectively, where the ambient space will always be clear from the context. 
For a surface $F$ in a $3$-manifold and a compressing disk $\Delta$ (or a union $\Delta$ of mutually disjoint compressing disks), we denote by $F / \Delta$ the surface obtained by compressing $F$ along $\Delta$.

\section{Preliminaries}\label{sec2}

Let $(V, W; \Sigma)$ be a genus-$g$ Heegaard splitting of $S^3$ with $g \ge 3$.
Let $P$ and $Q$ be reducing spheres for $\Sigma$. 
When the vertices of $\mathcal{R}(\Sigma)$ represented by $P$ and $Q$ belong to the same component of $\mathcal{R}(\Sigma)$, we simply write $P \sim Q$.
An essential disk $D$ in $V$ is called a {\em primitive disk} if there exists an essential disk $E$ in $W$ such that $\partial D$ and $\partial E$ intersect transversely in a single point.
Such a disk $E$ is called a {\em dual disk} of $D$.
Of course, $E$ is also a primitive disk, and $D$ is a dual disk of $E$.
The pair of a primitive disk $D$ and its dual disk $E$ is called a {\em dual pair}, and is denoted by $(D, E)$ or $(E, D)$.
The $2$-sphere $\partial \N(D \cup E)$ is a reducing sphere for $\Sigma$, which we will call a reducing sphere {\em associated with} the dual pair $(D, E)$.
Any primitive disk is necessarily non-separating, and in fact, is a member of a non-separating weak reducing pair since, for a dual pair $(D, E)$, we can choose a suitable non-separating compressing disk for $\Sigma$ outside $\N(D \cup E)$ that forms a weak reducing pair with $D$ (or $E$).

\begin{lemma}\label{lem2.1}
Let $D$ be a primitive disk in $V$.
Let $\Delta$ be an essential disk in $W$ that is disjoint from $D$.
Then there exists a dual disk $E$ of $D$ that is disjoint from $\Delta$. 
\end{lemma}

\begin{proof}
Take a dual disk $E$ of $D$ so that $| E \cap \Delta |$ is minimal.
Suppose that $| E \cap \Delta | > 0$.
Let $\Delta_0$ be an outermost subdisk of $\Delta$ cut off by an arc $\alpha$ of $E \cap \Delta$ that is outermost in $\Delta$. 
The arc $\alpha$ cuts $E$ into two subdisks $E_1$ and $E_2$.
Let $E_1$ be the one that has the point $\partial D \cap \partial E$. 
Then $E' = E_1 \cup \Delta_0$ is a dual disk of $D$.
By a slight isotopy of $E'$, we have $| E' \cap \Delta | < | E \cap \Delta |$ since at least the arc $\alpha$ no longer counts.
It contradicts the minimality of $| E \cap \Delta |$, and hence $| E \cap \Delta | = 0$.
\end{proof}

Two dual pairs $(D, E)$ and $(D', E')$ are said to be {\em p-connected} if there exists a sequence of dual pairs $(D, E) = (D_1, E_1), (D_2, E_2), \ldots, (D_n, E_n) = (D', E')$ such that $D_i \cup E_i$ and $D_{i+1} \cup E_{i+1}$ are disjoint for each $i \in \{1, 2, \ldots, n-1\}$.
The following lemma, shown in \cite{CKL24}, says that any two dual pairs that share a common member are p-connected.

\begin{lemma}\cite[Lemma 2.2]{CKL24}\label{lem2.2}
Let $D$ be a common dual disk of two primitive disks $E$ and $E'$.
Then the dual pairs $(D, E)$ and $(D, E')$ are p-connected.
\end{lemma}

The following is a consequence of Lemmas \ref{lem2.1} and \ref{lem2.2}.

\begin{lemma}\label{lem2.3}
Let $D$ be a primitive disk in $V$.
Let $P$ and $Q$ be reducing spheres for $\Sigma$ that are disjoint from $D$.
Then $P \sim Q$.
\end{lemma}

\begin{proof}
By Lemma \ref{lem2.1}, there exists a dual disk $E$ of $D$ such that $E$ is disjoint from $P \cap W$, and there exists a dual disk $E'$ of $D$ such that $E'$ is disjoint from $Q \cap W$.
Then $D$ is a common dual disk of $E$ and $E'$.
By Lemma \ref{lem2.2}, there exists a sequence of dual pairs $(D, E) = (D_1, E_1), (D_2, E_2), \ldots, (D_n, E_n) = (D, E')$ such that $D_i \cup E_i$ and $D_{i+1} \cup E_{i+1}$ are disjoint for each $i \in \{1, 2, \ldots, n-1\}$.
Thus we can choose a reducing sphere $P_i$ for $\Sigma$ associated with $(D_i, E_i)$ for each $i \in \{1, 2, \ldots, n\}$ so that $P_i$ is disjoint from $P_{i+1}$ for each $i \in \{1, 2, \ldots, n-1\}$. 
Furthermore, the reducing spheres $P_1$ and $P_n$ can be chosen to be disjoint from (or isotopic to) $P$ and $Q$, respectively, and hence we have $P \sim Q$.
\end{proof}

We will use the following lemma frequently in Sections \ref{sec5} and \ref{sec6}.

\begin{lemma}\label{lem2.4}
Let $\mathcal{C}$ be a finite collection of mutually disjoint compressing disks for $\Sigma$.
Suppose that the surface obtained by compressing $\Sigma$ along the union of disks of $\mathcal{C}$ has a $2$-sphere component $S$.
Let $D_1, \ldots, D_n$ be the disks of $\mathcal{C}$ such that $D_i$ is contained in $V$ and at least one scar of $D_i$ lies in $S$ for each $i \in \{1, 2, \ldots, n\}$.
Let $E_1, \ldots, E_m$ be the disks of $\mathcal{C}$ such that $E_j$ is contained in $W$ and at least one scar of $E_j$ lies in $S$ for each $j \in \{1, 2, \ldots, m\}$.
Let $\gamma$ be a circle in $S$ that separates the scars from $D_1, \ldots, D_n$ and the scars from $E_1, \ldots, E_m$.
Then there exists a reducing sphere $P$ for $\Sigma$ such that $\gamma = P \cap \Sigma$ and $P$ separates $\bigcup^n_{i=1} D_i$ and $\bigcup^m_{j=1} E_j$.
That is, $\bigcup^n_{i=1} D_i$ and $\bigcup^m_{j=1} E_j$ are disjoint from $P$, and they are contained in different components cut off by $P$. 
\end{lemma}

\begin{proof}
If we recover the original Heegaard surface $\Sigma$, then $\gamma$ becomes an essential separating circle in $\Sigma$, that bounds essential separating disks $D$ and $E$ in $V$ and $W$, respectively.
Furthermore, we can choose such disks $D$ and $E$ so that they are disjoint from $\bigcup^n_{i=1} D_i$ and $\bigcup^m_{j=1} E_j$ respectively.
Then the union of $D$ and $E$ is the desired reducing sphere $P$.
\end{proof}

We end the section with a special property of the genus-$3$ Heegaard splitting of $S^3$.

\begin{lemma}\label{lem2.5}
For a genus-$3$ Heegaard splitting $(V, W; \Sigma)$ of $S^3$, if $D - E$ is a non-separating weak reducing pair for $\Sigma$, then at least one of $D$ and $E$ is a primitive disk.
\end{lemma}

\begin{proof}
It is shown in \cite[Lemma 2.1]{CKL25} that there always exists a reducing sphere $P$ for $\Sigma$ that separates such a pair $D - E$ (this is a special property of the Heegaard splitting of $S^3$ only of genus $2$ or $3$).
The reducing sphere $P$ cuts off solid tori $V_1$ and $W_1$ from $V$ and $W$, respectively.
Then we have only two cases, either $D$ is a meridian disk of $V_1$ or $E$ is a meridian disk of $W_1$.
The former case, $D$ turns out to be a primitive disk with a dual disk that is a meridian disk of the solid torus $W_1$ disjoint from $P$.
In the latter case, by a symmetric argument, we see that $E$ is primitive.
\end{proof}

\section{On a $\partial$-version of Haken's lemma}\label{sec3}

Let $F$ be a closed orientable surface. 
A {\em compression body} $V$ is a $3$-manifold obtained from $F \times [0, 1]$ by attaching $2$-handles to $F \times\{ 1 \}$ and then by capping off any resulting $2$-sphere boundary components with $3$-balls.
The surface $F \times \{ 0 \}$ is denoted by $\partial_{+} V$, and $\partial V - \partial_{+} V$ is denoted by $\partial_{-} V$.
If $\partial_{-} V = \emptyset$, then $V$ is a handlebody.

Let $V$ be a compression body with $\partial_{-} V \ne \emptyset$ and suppose that $\chi(\partial_{-} V) > \chi(\partial_{+} V)$.
A collection $\mathcal{D}$ of pairwise disjoint compressing disks $D_1, \ldots, D_n$ for $\partial_{+} V$ is called a {\em complete disk system} for $V$ if the result of cutting $V$ along $\bigcup^n_{i=1} D_i$ is $\partial_{-} V \times I$.
If $V$ is a handlebody, then a {\em complete disk system} for $V$ is a collection $\mathcal{D}$ of pairwise disjoint compressing disks $D_1, \ldots, D_n$ for $\partial V = \partial_+V$ such that the result of cutting $V$ along $\bigcup^n_{i=1} D_i$ is a $3$-ball.

As a special case, when $\chi(\partial_{-} V) = \chi(\partial_{+} V) + 2$, it is not difficult to see that $V$ has a unique complete disk system consisting only of a single disk, say $\mathcal{D} = \{ D_1 \}$, up to isotopy.
In this case, $V$ is a compression body obtained by attaching a single $2$-handle to $F \times \{1\}$. 
If $D$ is a core disk of the $2$-handle, then the disk $D_1 = D \cup (\partial D \times [0, 1])$ forms the unique complete disk system $\mathcal D$ for $V$.
See Figure \ref{fig1} for examples.
In Figure \ref{fig1} (a), $D_1$ is the unique non-separating compressing disk for $\partial_{+} V$ up to isotopy, and in Figure \ref{fig1} (b), $D_1$ is the unique compressing disk for $\partial_{+} V$ up to isotopy, that is separating.
Note that, for Figure \ref{fig1} (a), there are infinitely many separating compressing disks for $\partial_{+} V$ up to isotopy, that can be obtained by banding two parallel copies of $D_1$ along arcs in $\partial_{+} V$.
But such a separating disk cannot form a complete disk system for $V$ by definition.

\begin{figure}[!hbt]
\includegraphics[width=14cm,clip]{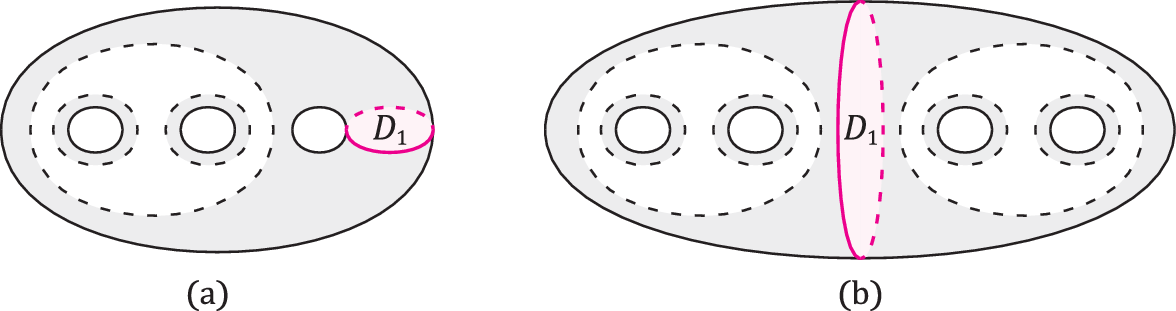}
\caption{Compression bodies with their unique complete disk systems.}
\label{fig1}
\end{figure}

The definition of Heegaard splitting for closed $3$-manifolds can be generalized to compact $3$-manifolds with boundary.
A {\em Heegaard splitting} $(V, W; \Sigma)$ of a compact $3$-manifold $M$ is a decomposition of $M$ into two compression bodies $V$ and $W$ such that $M = V \cup W$ and $V \cap W = \partial_{+} V = \partial_{+} W = \Sigma$ and $\partial M = \partial_{-} V \cup \partial_{-} W$.
A {\em $\partial$-reducing disk} of a $3$-manifold $M$ is a compressing disk for $\partial M$. 
The following is a simplified version (for our purpose) of a lemma due to Casson and Gordon, which is known as a $\partial$-version of Haken's lemma.

\begin{lemma}[A simplified version of Lemma $1.1$ of \cite{Casson-Gordon}]\label{lem3.1}
Let $(V, W; \Sigma)$ be a Heegaard splitting of an irreducible $3$-manifold $M$.
Let $D$ be a $\partial$-reducing disk of $M$ with $\partial D \subset \partial_{-}W$.
Then there exists a $\partial$-reducing disk $D^{\ast}$ of $M$ such that

\begin{enumerate}
\item[(i)] $D^{\ast}$ is isotopic to $D$ in $M$;
\item[(ii)] $D^{\ast}$ intersects $\Sigma$ in a single circle;
\item[(iii)] there exist complete disk systems $\mathcal{D}$ and $\mathcal{E}$ for $V$ and $W$, respectively, such that every disk of $\mathcal{D}$ and $\mathcal{E}$ is disjoint from $D^{\ast}$.
\end{enumerate}
\end{lemma}

\section{Topologically minimal surfaces}\label{sec4}

Let $F$ be a separating surface embedded in a $3$-manifold $M$, not necessarily a Heegaard surface.
The (non-separating) weak reducing disks and pairs can be defined for $F$ in $M$ in the same way. 
The surface $F$ is said to be {\em weakly reducible} if there exists a weak reducing pair for $F$.
The surface $F$ is said to be {\em strongly irreducible} if $F$ is compressible to both sides and $F$ is not weakly reducible.
The surface $F$ is called a {\em critical surface} if the set of compressing disks for $F$ can be partitioned into two sets $\mathcal{C}_0 \sqcup \mathcal{C}_1$ satisfying the following properties (see \cite{Bachman1} and \cite{Bachman3}):

\begin{itemize}
\item For each $i \in \{ 0, 1 \}$, there exist disjoint compressing disks $D_i, E_i \in \mathcal{C}_i$ on opposite sides of $F$, i.e. $D_i - E_i$ is a weak reducing pair.
\item For any compressing disks $D_i \in \mathcal{C}_i$ and $E_{1-i} \in \mathcal{C}_{1-i}$ ($i \in \{ 0, 1 \}$) on opposite sides of $F$, the disks $D_i$ and $E_{1-i}$ intersect each other.
\end{itemize}

A sequence of weak reducing disks $\Delta_1, \Delta_2, \ldots, \Delta_n$ for $F$, denoted by $\Delta_1 - \Delta_2 - \cdots - \Delta_n$, is called a {\em weak reducing sequence} for $F$ connecting $\Delta_1$ to $\Delta_n$, if $\Delta_i - \Delta_{i+1}$ is a weak reducing pair for each $i \in \{1, 2, \ldots, n-1\}$.
The following lemma is a generalization of Lemma $8.5$ in \cite{Bachman2}, which gives a sufficient condition for the surface $F$ to be critical.

\begin{lemma}\label{lem4.1}
Suppose there are two weak reducing disks $D$ and $D'$ for $F$ such that there is no weak reducing sequence connecting $D$ to $D'$.
Then $F$ is a critical surface.
\end{lemma}

We remark that Lemma $8.5$ in \cite{Bachman2} is stated for a Heegaard surface but the same proof identically works for any separating surface embedded in a $3$-manifold.

\smallskip

The {\em disk complex} $\mathcal{D}(F)$ of the surface $F$ is a simplicial complex defined as follows.
The vertices of $\mathcal{D}(F)$ are the isotopy classes of compressing disks for $F$, and a collection of $k+1$ distinct vertices spans a $k$-simplex if there are representative disks of the vertices that are pairwise disjoint.

The surface $F$ is said to be {\em topologically minimal} if $\mathcal{D}(F) = \emptyset$ or $\mathcal{D}(F)$ is not contractible, that is, $\pi_i(\mathcal{D}(F)) \ne 1$ for some $i$.
The {\em topological index} of a topologically minimal surface $F$, simply the {\em index} of $F$, is defined to be $0$ if $\mathcal{D}(F) = \emptyset$, and otherwise, to be the smallest $i$ such that $\pi_{i-1}(\mathcal{D}(F)) \ne 1$. 
The following lemma follows from \cite[Theorems 2.3 and 2.5]{Bachman3}. 

\begin{lemma}
\label{lem: strong irreducibility and criticality}
Let $F$ be a separating surface embedded in a $3$-manifold $M$. 
\begin{enumerate}
\item $F$ is strongly irreducible if and only if it is an index $1$ topologically minimal surface. 
\item $F$ is critical if and only if it is an index $2$ topologically minimal surface.
\end{enumerate}
\end{lemma}

\smallskip

For the genus-$g$ Heegaard surface $\Sigma$ in $S^3$ with $g \geq 2$, it is shown in \cite[Appendix A]{CKL24} that $\pi_1(\mathcal{D}(\Sigma)) = 1$, and hence $\Sigma$ is not a critical surface.
Thus Lemma \ref{lem4.1} tells us that, for any two weak reducing disks for $\Sigma$, there always exists a weak reducing sequence for $\Sigma$ connecting them.

\smallskip

The following is a key property of a topologically minimal surface, shown in \cite{Bachman3}.

\begin{lemma}\cite[Theorem 4.3]{Bachman3}\label{lem4.3}
Let $F$ be a closed surface embedded in a $3$-manifold $M$, and let $F$ separate $M$ into $V$ and $W$.
Let $F_V$ be a surface obtained from $F$ by a maximal sequence of compressions into $V$.
If $F$ is a topologically minimal surface in $M$, then $F_V$ is incompressible in $M$.
\end{lemma}

Finally, we end the section with a lemma, a direct consequence of Lemma \ref{lem4.3}, which will be used in Section \ref{sec6}.

\begin{lemma}\label{lem4.4}
Any closed topologically minimal surface in $S^3$ is a Heegaard surface.
\end{lemma}

\begin{proof}
Let $F$ be a closed topologically minimal surface in $S^3$, and let $F$ separate $S^3$ into $V$ and $W$.
Let $F_V$ and $F_W$ be surfaces obtained from $F$ by a maximal sequence of compressions into $V$ and $W$ respectively.
Since $F_V$ is incompressible in $S^3$ by Lemma \ref{lem4.3}, $F_V$ consists only of $2$-sphere components.
If we cap off the $2$-spheres with $3$-balls, it turns out that $V$ is a handlebody.
Similarly, $W$ is also a handlebody, and hence $F$ is a Heegaard surface.
\end{proof}

\section{Non-separating weak reducing disks for $\Sigma$}\label{sec5}

In general, a simple closed curve $C$ in the boundary of a handlebody $V$ is called a {\em primitive curve} if there exists an essential disk $D$ in $V$ such that $\partial D$ and $C$ intersect transversely in a single point.
Again, such a disk $D$ is called a {\em dual disk} of $C$.

\begin{lemma}[A simplified version of Theorem $1$ of \cite{Gordon}]\label{lem5.1}
Let $C$ be a simple closed curve in the boundary of a handlebody $V$ such that the result of adding a $2$-handle to $V$ along $C$ is also a handlebody.
Then $C$ is a primitive curve.
\end{lemma}

The two lemmas in the following will be used in Section \ref{sec7} to construct a weak reducing sequence consisting only of non-separating disks for the genus-$4$ Heegaard surface in $S^3$.

\begin{lemma}\label{lem5.2}
Let $(V, W; \Sigma)$ be the genus-$4$ Heegaard splitting of $S^3$.
Let $D - E - D'$ be a weak reducing sequence for $\Sigma$, where $D, D' \subset V$ are not isotopic.
Suppose that $E \subset W$ is a separating disk.
Then there exists a new weak reducing sequence $D - E' - D'$ or $D - E' - D'' - E'' - D'$ for $\Sigma$ such that $E'$, $D''$, $E''$ are all non-separating.
\end{lemma}

\begin{proof}
Let $W_1$ and $W_2$ be the two handlebodies obtained by cutting $W$ along $E$.
If both of $\partial D $ and $\partial D'$ lie in $\partial W_1$ or in $\partial W_2$, say in $\partial W_1$, we can take a non-separating disk $E'$ in $W_2$ so that $D - E' - D'$ is a desired weak reducing sequence for $\Sigma$.
So we may assume that $\partial D \subset \partial W_1$ and $\partial D' \subset \partial W_2$.
If $\partial D$ is inessential in $\partial W_1$ or $\partial D'$ is inessential in $\partial W_2$, say $\partial D$ is inessential, then $\partial D$ is parallel to $\partial E$ in $\Sigma$.
In this case, we isotope $\partial D$ into $\partial W_2$.
Then both $\partial D$ and $\partial D'$ lie in $\partial W_2$.
We take a non-separating disk $E'$ in $W_1$ as above so that $D - E' - D'$ is a desired weak reducing sequence for $\Sigma$.
From now on, we assume that $\partial D$ and $\partial D'$ are essential in $\partial W_1$ and $\partial W_2$, respectively.
We have two cases, according to the genus of $W_1$ and $W_2$.

\vspace{0.2cm}

\noindent Case $1$. Either $W_1$ or $W_2$ is a solid torus.

\vspace{0.2cm}

We may assume that $W_1$ is a solid torus.
Since $\partial D$ is essential in $\partial W_1$, by compressing $\partial W_1$ along $D$, we get a $2$-sphere containing two scars of $D$ and one scar of $E$.
By Lemma \ref{lem2.4}, we can take a reducing sphere $P$ for $\Sigma$ separating $D$ and $E$.
Let $V'_1$ and $W'_1$ be solid tori cut off by $P$ from $V$ and $W$, respectively, and let $V'_2$ and $W'_2$ be genus-$3$ handlebodies cut off by $P$ from $V$ and $W$, respectively.
We take disjoint essential non-separating disks $E' \subset W'_2$ and $D'' \subset V'_2$ disjoint from $P$, and a meridian disk $E'' \subset W'_1$ disjoint from $P$.
Then $D - E' - D'' - E'' - D'$ is a desired weak reducing sequence for $\Sigma$.

\vspace{0.2cm}

\noindent Case $2$. Both $W_1$ and $W_2$ are genus-$2$ handlebodies.

\vspace{0.2cm}

Consider the surface $\partial W_1 / D$ first.
If $\partial D$ is separating in $\partial W_1$, then $\partial W_1 / D$ consists of two torus components.
Let $\Sigma_1$ be one of the two components which contains one scar of $D$ and one scar of $E$.
If $\partial D$ is non-separating in $\partial W_1$, then $\partial W_1 / D$ is a single torus, containing two scars from $D$ and one scar of $E$.
Let $\Sigma_1$ be $\partial W_1 / D$ itself in this case.
In both cases, $\Sigma_1$ is a torus in $S^3$, which is compressible.
Thus, at least one of the two manifolds bounded by $\Sigma_1$ is a solid torus.
We simply say that $\Sigma_1$ is {\em compressible in $W$-side} if the manifold bounded by $\Sigma_1$ and containing $W_1$ is a solid torus.

In the same way, we have a torus $\Sigma_2$, which is the component of $\partial W_2 / D'$ containing one scar of $D'$ and one scar of $E$ if $\partial D'$ is separating in $\partial W_2$, and is $\partial W_2 / D'$ itself if $\partial D'$ is non-separating in $\partial W_2$.
Again, we say that $\Sigma_2$ is {\em compressible in $W$-side} if the manifold bounded by $\Sigma_2$ and containing $W_2$ is a solid torus.
We divide Case $2$ into two subcases, as follows.

\vspace{0.2cm}

\noindent Case $2$-$1$. 
At least one of $\Sigma_1$ and $\Sigma_2$ is compressible in $W$-side.

\vspace{0.2cm}

It suffices to consider the case where $\Sigma_1$ is compressible in $W$-side. 
The essential curve $\partial D$ is separating or non-separating in $\partial W_1$.

First, assume that $\partial D$ is separating in $\partial W_1$.
Let $W_0$ be the solid torus bounded by $\Sigma_1$, which contains $W_1$.
Let $\Delta$ be a meridian disk of the solid torus component of $\overline{V - \N(D)}$, taken to be disjoint from $D$.
Adding a $2$-handle $\N(\Delta)$ to the genus-$2$ handlebody $W_1$ results in a manifold that is isotopic to the solid torus $W_0$.
Thus, $\partial \Delta$ is a primitive curve in $\partial W_1$ by Lemma \ref{lem5.1}. 
So there exists a dual disk $\Delta'$ of $\partial\Delta$ in $W_1$, and furthermore, we can take the dual disk $\Delta'$ to be disjoint from $E$.
A reducing sphere $P$ for $\Sigma$ associated with $(\Delta, \Delta')$ separates $\Delta \cup \Delta'$ and $E$.
We take a meridian disk $E''$, in the solid torus component of $\overline{W_1 - \N(P \cap W)}$ containing one scar of $E$, such that $E''$ is disjoint from $P \cap W$ and $E$.
We also take a non-separating essential disk $E' \subset W_2$ disjoint from $E$.
Then $D - E' - \Delta - E'' - D'$ is a desired weak reducing sequence for $\Sigma$.
See Figure \ref{fig2}.

\begin{figure}[!hbt]
\includegraphics[width=13cm,clip]{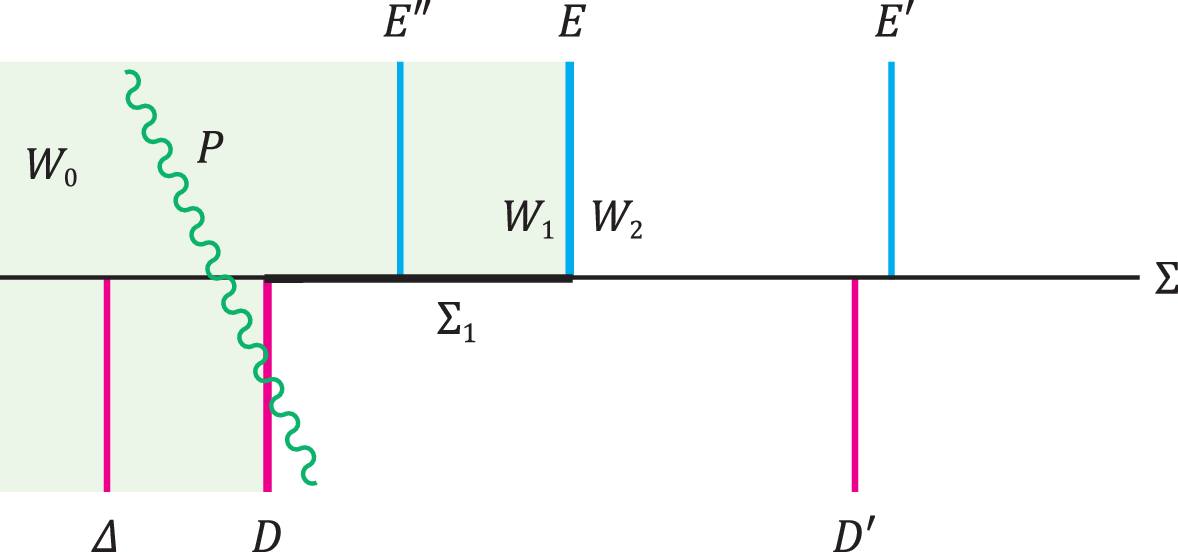}
\caption{A weak reducing sequence $D - E' - \Delta - E'' - D'$ with $E'$, $\Delta$, $E''$ non-separating.}
\label{fig2}
\end{figure}

Next, assume that the curve $\partial D$ is non-separating in $\partial W_1$.
Then the torus $\Sigma_1$ contains two scars from $D$ and one scar of $E$, and $\Sigma_1$ is the boundary of the solid torus $W_1 \cup \N(D)$.
By an argument similar to the above, we see that $\partial D$ is a primitive curve in $\partial W_1$, thus, in $\partial W$.
Thus there exists a reducing sphere $P$ for $\Sigma$ separating $D$ and $E$.
We take a meridian disk $E'$, in the solid torus component of $\overline{W_1 - \N(P \cap W)}$ containing one scar of $E$, such that $E'$ is disjoint from $P \cap W$ and $E$.
Then $D - E' -D'$ is a desired weak reducing sequence.

\vspace{0.2cm}

\noindent Case $2$-$2$. 
None of $\Sigma_1$ and $\Sigma_2$ is compressible in $W$-side.

\vspace{0.2cm}

Let $V'_0$ be the genus-$2$ handlebody component of $\overline{V - \N(D \cup D')}$, which contains $\partial E$.
Let $V_0$ be a slightly shrunken handlebody in the interior of $V'_0$ such that $\overline{V'_0 - V_0}$ is homeomorphic to $\partial V_0 \times [0, 1]$.
Let $W_0$ be the compression body obtained by attaching a $2$-handle $N(E)$ to $\overline{V'_0 - V_0}$.
Then $(V_0, W_0; \partial V_0)$ is a Heegaard splitting of $M_0 = V_0 \cup W_0$.
The compression body $W_0$ has the unique complete disk system $\{ \overline{E} \}$, where $\overline{E} = E \cup (\partial E \times [0, 1])$, i.e. $\overline{E}$ is a disk obtained by extending $E$ slightly.
Let $D_0$ be a compressing disk for the torus $\Sigma_1$, that is, $D_0$ is a meridian disk of the solid torus bounded by the torus $\Sigma_1$.

Suppose first that $D_0 \cap \Sigma_2 \ne \emptyset$, and assume that they intersect minimally.
Let $\Delta$ be an innermost disk in $D_0$ cut off by $\Sigma_2$.
Then $\Delta$ is a compressing disk for $\Sigma_2$.
Since $\Sigma_2$ is not compressible in $W$-side and since $\Delta$ is disjoint from $\Sigma_1$, $\Delta$ must lie in $M_0$.
Note that the compression body $W_0$ has the unique complete disk system $\{ \overline{E} \}$, thus by Lemma \ref{lem3.1}, there exists a compressing disk $\Delta'$ for $\Sigma_2$ in $M_0$ that intersects $\partial V_0$ in a single circle and is disjoint from $\overline{E}$.
Moreover, the disk $\Delta'$ can be chosen to be disjoint from the disk $D'$ by a slight isotopy.
See Figure \ref{fig3}.

Since $\Sigma_2 / \Delta'$ is a $2$-sphere, there exists a reducing sphere $P$ for $\Sigma$ separating $E$ and $\Delta' \cup D'$ by Lemma \ref{lem2.4}.
The sphere $P$ cuts $V$ into two genus-$2$ handlebodies $V'_1$ and $V'_2$, and cuts $W$ into two genus-$2$ handlebodies $W'_1$ and $W'_2$.
We may assume that $V'_1 \cup W'_1$ and $V'_2 \cup W'_2$ lie in the opposite sides of $P$, and that $D \subset V'_1$, $E \subset W'_1$ and $D' \subset V'_2$.
We take a non-separating essential disk $E' \subset W'_2$ disjoint from $P$, and disjoint non-separating essential disks $D'' \subset V'_1$ and $E'' \subset W'_1$ both disjoint from $P$.
Then $D - E' - D'' - E'' - D'$ is a desired weak reducing sequence.

\begin{figure}[!hbt]
\includegraphics[width=12cm,clip]{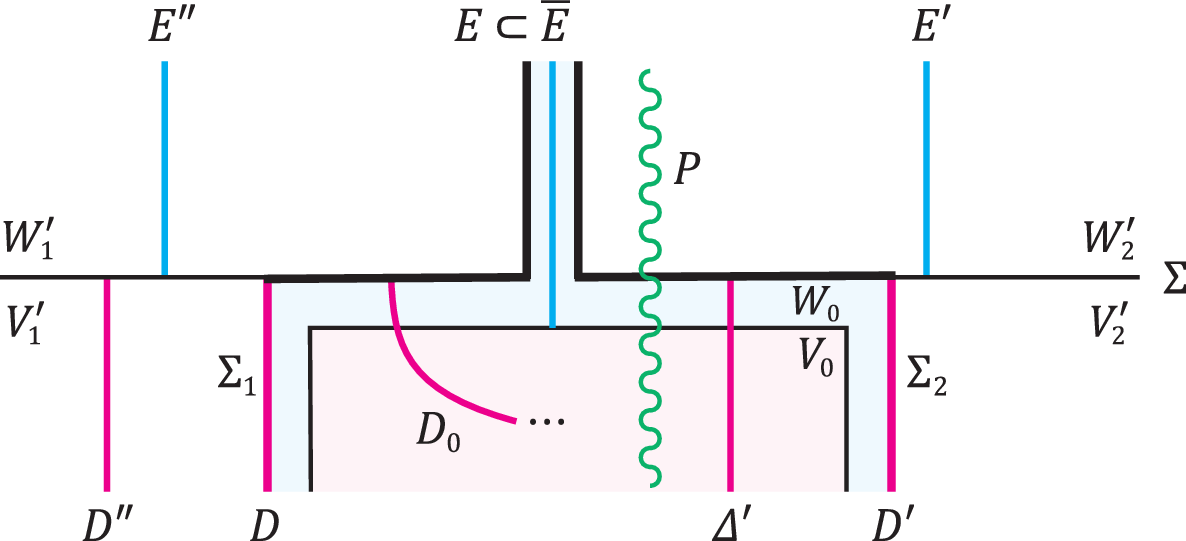}
\caption{A reducing sphere $P$ and a weak reducing sequence $D - E' - D'' - E'' - D'$.}
\label{fig3}
\end{figure}

Suppose next that $D_0 \cap \Sigma_2 = \emptyset$.
Since $\Sigma_1$ is not compressible in $W$-side, the compressing disk $D_0$ for $\Sigma_1$ must lie in $M_0$.
Thus, using $D_0$ instead of $\Delta$, we can find a weak reducing sequence $D - E' - D'' - E'' - D'$ for $\Sigma$, where $E'$, $D''$, $E''$ are all non-separating.
\end{proof}

Similarly, switching the roles of the disks, we have the following.

\begin{lemma}\label{lem5.3}
Let $(V, W; \Sigma)$ be the genus-$4$ Heegaard splitting of $S^3$.
Let $E - D - E'$ be a weak reducing sequence for $\Sigma$, where $E, E' \subset W$ are not isotopic.
Suppose that $D \subset V$ is a separating disk.
Then there exists a new weak reducing sequence $E - D' - E'$ or $E - D' - E'' - D'' - E'$ for $\Sigma$ such that $D'$, $E''$, $D''$ are all non-separating.
\end{lemma}

\section{Proof of Lemma \ref{lem1.1}}\label{sec6}

Let $(V, W; \Sigma)$ be the genus-$4$ Heegaard splitting of $S^3$, and let $D - E$ be a non-separating weak reducing pair for $\Sigma$ such that $D \subset V$ and $E \subset W$.
Suppose that the surface $\Sigma_0 = \Sigma / (D \cup E)$ is compressible to both sides.

\begin{proof}[Proof of Lemma \ref{lem1.1}]
We want to find a reducing sphere $P$ that separates the pair $D - E$. 

If one of $D$ and $E$, say $D$, is primitive, then by Lemma \ref{lem2.1} we have a dual disk $\Delta \subset W$ of $D$ that is disjoint from $E$.
Then a reducing sphere $P$ for $\Sigma$ associated with $(D, \Delta)$ separates the pair $D - E$ as desired.
In this case, the conclusion follows without using the condition for the surface $\Sigma_0$.
From now on, we assume that both $D$ and $E$ are non-primitive.

\medskip

Each of $\partial D$ and $\partial E$ is non-separating in $\Sigma$.
If $\partial D \cup \partial E$ is separating in $\Sigma$, then each component of $\Sigma_0$ is a closed non-separating surface in $S^3$, a contradiction.
Thus $\partial D \cup \partial E$ is non-separating in $\Sigma$, and hence $\Sigma_0$ is a genus-$2$ surface.
Let $V_0 = \overline{V - \N(D)} \cup \N(E)$ and $W_0 = \overline{W - \N(E)} \cup \N(D)$.
Then $\Sigma_0 = \partial V_0 = \partial W_0$.
Since $E$ and $D$ are non-primitive, each of the manifolds $V_0$ and $W_0$ is not a handlebody by Lemma \ref{lem5.1}, and hence $\Sigma_0$ is not a Heegaard surface in $S^3$.
By Lemma \ref{lem4.4}, the surface $\Sigma_0$ is not topologically minimal.
Thus $\Sigma_0$ is not strongly irreducible by Lemma~\ref{lem: strong irreducibility and criticality}.
Since $\Sigma_0$ is compressible to both sides, it is weakly reducible.
Therefore, there exists a weak reducing pair $D_1 - E_1$ for $\Sigma_0$, where $D_1 \subset V_0$ and $E_1 \subset W_0$.
We may assume that $\partial D_1$ is disjoint from the two scars from $D$ in $\Sigma_0$ and that $\partial E_1$ is disjoint from the two scars from $E$ in $\Sigma_0$.
But, $D_1$ and $E_1$ might intersect the $2$-handles $\N(E)$ and $\N(D)$, respectively.

\medskip

We will show first that $D_1$ and $E_1$ can be isotoped to be disjoint from $\N(E)$ and $\N(D)$, respectively, so that they can be regarded as compressing disks for the original Heegaard surface $\Sigma$ (Claim \ref{claim1}).
The resulting disks might intersect each other in their boundaries, and then we will modify them to obtain compressing disks $D^{\ast}_1$ and $E^{\ast}_1$ for $\Sigma$ that are disjoint from each other (Claim \ref{claim2}).
Finally, we will find a reducing sphere $P$ separating the pair $D - E$ by investigating the curves $\partial D^{\ast}_1$ and $\partial E^{\ast}_1$ in the surface $\Sigma_0$.

\begin{claim}\label{claim1}
We can isotope the disk $D_1$ in $V_0$ to a disk $D^{\ast\ast}_1$ disjoint from $\N(E)$, and the disk $E_1$ in $W_0$ to a disk $E^{\ast\ast}_1$ disjoint from $\N(D)$, respectively.
\end{claim}

\begin{proof}[Proof of Claim 1]
Consider the case of the disk $D_1$ in $V_0$ first.
Let $\Sigma_1$ be a parallel copy of $\Sigma / D$ in the interior of $V_0$.
Then $\Sigma_1$ decomposes $V_0$ into two submanifolds $V_1$ and $W_1$, where $V_1$ is a (shrunken) genus-$3$ handlebody homeomorphic to $\overline{V - \N(D)}$, and $W_1 = \overline{V_0 -V_1}$ is a compression body, which is obtained by attaching a $2$-handle $\N(E)$ to the product of $\Sigma_1$ and a closed interval.
Therefore we have a Heegaard splitting $(V_1, W_1; \Sigma_1)$ of $V_0$.
Extending the disk $E$ to $\overline{E}$ so that $\partial \overline{E} \subset \Sigma_1$, we have the unique complete disk system $\{ \overline{E} \}$ for $W_1$.
The disk $D_1$ is a compressing disk for $\Sigma_0$ in $V_0$.
By Lemma \ref{lem3.1}, there exists a compressing disk $D^{\ast}_1$ for $\Sigma_0$ such that $D^{\ast}_1$ is isotopic to $D_1$ in $V_0$, and $D^{\ast}_1$ intersects $\Sigma_1$ in a single circle, and $D^{\ast}_1 \cap \overline{E} = \emptyset$.
The new disk $D^{\ast}_1$ still might intersect $\N(E)$.
We will isotope $D^{\ast}_1$ to a disk $D^{\ast\ast}_1$ disjoint from $\N(E)$.
See Figure \ref{fig5}.

Considering the $2$-handle $\N(E)$ as the product $E \times [0, 1]$ with $E = E \times \{\frac{1}{2}\}$, we have two copies $E \times \{ 0 \}$ and $E \times \{ 1 \}$ of the disk $E$ lying in $\Sigma_0$.
We may assume that $\partial D^{\ast}_1$ does not intersect $E \times \{ 0 \}$ and $E \times \{ 1 \}$ by an isotopy.
We extend the product $E \times [0, 1]$ to $\overline{E} \times [0, 1]$ so that $\overline{E} \times \{ 0 \} = E \times \{ 0 \}$, $\overline{E} \times \{ 1 \} = E \times \{ 1 \}$ and $\overline{E} \times \{ \frac{1}{2} \} = \overline{E}$ as described in Figure \ref{fig5}.
If $D^{\ast}_1 \cap (\overline{E} \times [0, 1]) = \emptyset$, let $D^{\ast\ast}_1$ be $D^{\ast}_1$, which is disjoint from $\N(E)$ and so we are done.

\begin{figure}[!hbt]
\includegraphics[width=14.5 cm,clip]{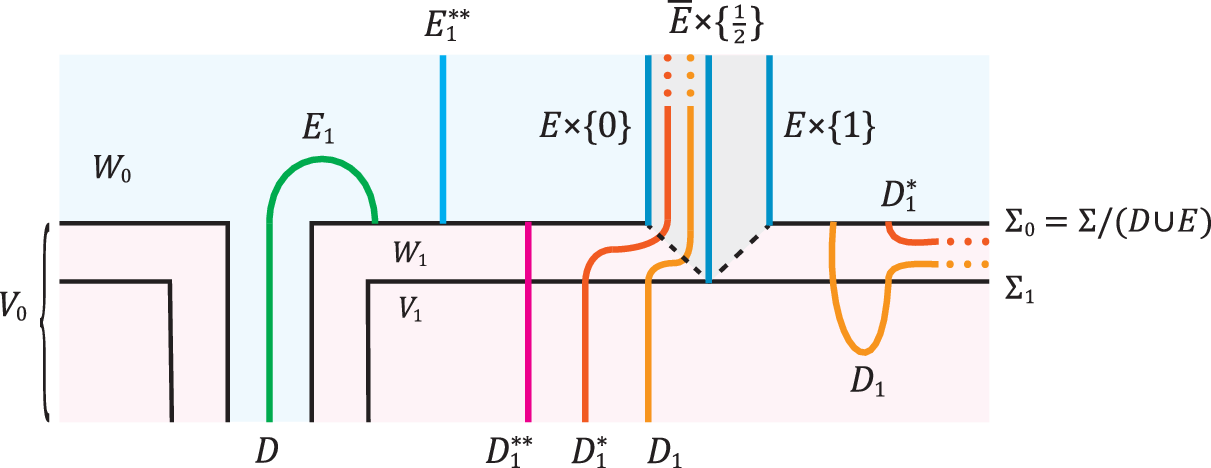}
\caption{The disks $D^{\ast\ast}_1$ in $V_0$ and $E^{\ast\ast}_1$ in $W_0$.}
\label{fig5}
\end{figure}

Suppose that $D^{\ast}_1 \cap (\overline{E} \times [0, 1]) \ne \emptyset$.
Let $A$ be the annulus $D^{\ast}_1 \cap W_1$ and let $F$ be the disk $D^{\ast}_1 \cap V_1$.
Then $D^{\ast}_1 = A \cup F$ and $\alpha = A \cap F$ is a circle in $\Sigma_1$.
Since $D^{\ast}_1$ is disjoint from $\overline{E} \times \{0\}$, $\overline{E} \times \{1\}$ and $\overline{E} \times \{ \frac{1}{2} \}$, the disk $D^{\ast}_1$ intersects $\overline{E} \times [0, 1]$ only in $\overline{E} \times (0, \frac{1}{2}) \cup \overline{E} \times (\frac{1}{2}, 1)$.

First assume that $D^{\ast}_1$ intersects $\overline{E} \times (0, \frac{1}{2})$.
If we cut $W_1$ along $\overline{E} \times \{ \frac{1}{2} \}$, then we get a $3$-manifold $W_2$ homeomorphic to the product $\Sigma_0 \times [0,1]$.
We regard $W_2$ as equipped with the product foliation whose leaves are the surfaces $\Sigma_0 \times \{t\}$. 
Let $p$ be the center point of $\overline{E} \times \{ \frac{1}{2} \}$.
Then the vertical arc $\{ p \} \times [ 0, \frac{1}{2} ]$ can be considered as a {\em transverse arc}, that is, the arc intersects each leaf of the foliation in a single point. 
Moreover, the annulus $A = D^{\ast}_1 \cap W_1$ can be also considered as a {\em transverse annulus} in $W_2$, that is, $A$ intersects each leaf in a single circle.
Thus, the annulus $A$ can be isotoped in $W_2$ to be disjoint from the transverse arc $\{ p \} \times [ 0, \frac{1}{2} ]$, fixing the boundary component $\alpha$ on $\Sigma_1$.
Since the product $\overline{E} \times [ 0, \frac{1}{2} ]$ can be regarded as a regular neighborhood of the arc $\{ p \} \times [ 0, \frac{1}{2} ]$ in $W_2$, the annulus $A$ can be isotoped further, fixing $\alpha$, to an annulus $A'$ that is disjoint from $\overline{E} \times ( 0, \frac{1}{2} )$.
The resulting disk $A' \cup F$ is isotopic to $D^{\ast}_1$, and is disjoint from $\overline{E} \times (0, \frac{1}{2})$.

If the disk $A' \cup F$ is disjoint from $\overline{E} \times (\frac{1}{2}, 1)$, then let $D^{\ast\ast}_1$ be $A' \cup F$.
If the disk $A' \cup F$ still intersects $\overline{E} \times (\frac{1}{2}, 1)$, then, by a similar argument, we eventually find a disk $D^{\ast\ast}_1$, which is isotopic to $D^{\ast}_1$ and even disjoint from $\overline{E} \times (\frac{1}{2}, 1)$. 
The disk $D^{\ast\ast}_1$ is disjoint from $\overline{E} \times [0, 1]$, and hence from $\N(E)$ as desired.

\medskip

By a symmetric argument, there exists a compressing disk $E^{\ast\ast}_1$ in $W_0$ such that $E^{\ast\ast}_1$ is isotopic to $E_1$ in $W_0$, and $E^{\ast\ast}_1$ is disjoint from $\N(D)$.
\end{proof}

For convenience, we change the notations $D^{\ast\ast}_1$ and $E^{\ast\ast}_1$ to $D^{\ast}_1$ and $E^{\ast}_1$, respectively.
Now, the new disks $D^{\ast}_1$ and $E^{\ast}_1$ can be regarded as essential disks in $V$ and $W$, respectively, and they are disjoint from $D \cup E$.
But, $D^{\ast}_1$ and $E^{\ast}_1$ might intersect each other on their boundaries.

\begin{claim}\label{claim2}
After modifying $D^{\ast}_1$ and $E^{\ast}_1$ by taking band sums with $D$ and $E$, respectively, and then applying an isotopy, they become disks in $V$ and $W$ that are disjoint from each other.
\end{claim}

\begin{proof}[Proof of Claim 2]
Suppose that $\partial D^{\ast}_1 \cap \partial E^{\ast}_1 \ne \emptyset$.
The disk $D^{\ast}_1$ is isotopic to $D_1$ in $V_0$, and the disk $E^{\ast}_1$ is isotopic to $E_1$ in $W_0$.
Since $D_1 - E_1$ is a weak reducing pair for $\Sigma_0$, we have $\partial D_1 \cap \partial E_1 = \emptyset$, and hence the curves $\partial D^{\ast}_1$ and $\partial E^{\ast}_1$ are not intersecting minimally in $\Sigma_0$.
By \cite[Proposition 1.7]{Farb-Margalit}, there exists a bigon disk in $\Sigma_0$ bounded by the union of a subarc of $\partial D^{\ast}_1$ and a subarc of $\partial E^{\ast}_1$.

Let $\Delta$ be an innermost such bigon disk.
The disk $\Delta$ contains some scars of $D$ or $E$.
Isotope $\partial D^{\ast}_1$ across the scar(s) of $D$ in $\Delta$, if any, and $\partial E^{\ast}_1$ across the scar(s) of $E$ in $\Delta$, if any.
These isotopies correspond to a band sum of $D^{\ast}_1$ and $D$ in $V$ and a band sum of $E^{\ast}_1$ and $E$ in $W$, respectively.
Further isotope $\partial D^{\ast}_1$ and $\partial E^{\ast}_1$ to remove $\Delta$.
See Figure \ref{fig6} for an example.
The number of intersection points $| \partial D^{\ast}_1 \cap \partial E^{\ast}_1 |$ is reduced by these isotopies.
By repeating this process, we make $\partial D^{\ast}_1$ and $\partial E^{\ast}_1$ disjoint.
\end{proof}

\begin{figure}[!hbt]
\includegraphics[width=14 cm,clip]{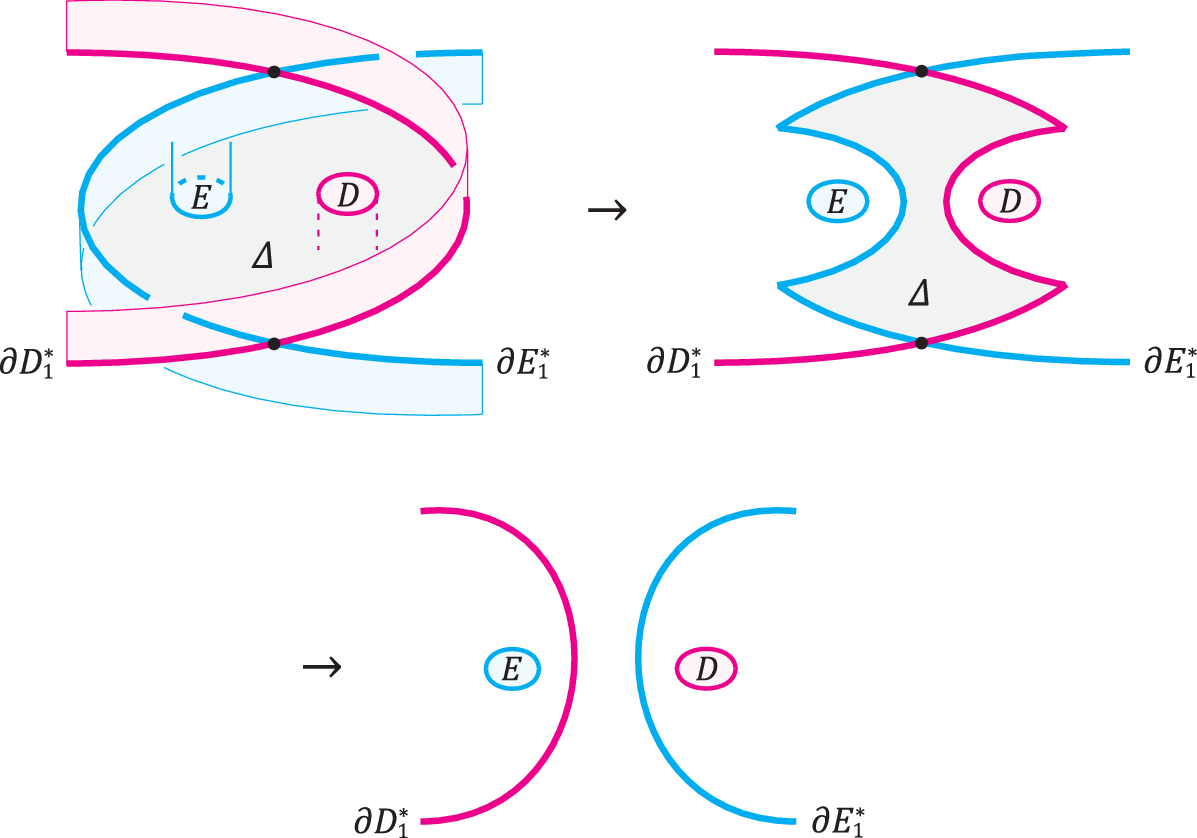}
\caption{Isotopies of $\partial D^{\ast}_1$ and $\partial E^{\ast}_1$.}
\label{fig6}
\end{figure}

By Claim \ref{claim2}, we may assume that $\partial D^{\ast}_1$ and $\partial E^{\ast}_1$ are disjoint essential curves in $\Sigma_0$.
Now what we have obtained are the disjoint compressing disks $D^{\ast}_1 \subset V$ and $E^{\ast}_1 \subset W$ for $\Sigma$ such that 

\begin{itemize}
\item $D^{\ast}_1$ is not isotopic to $D$ in $V$ and $E^{\ast}_1$ is not isotopic to $E$ in $W$, and
\item each of $D^{\ast}_1$ and $E^{\ast}_1$ is disjoint from $D \cup E$.
\end{itemize}

Finally, to find a reducing sphere for $\Sigma$ separating the pair $D - E$, we will investigate the curves $\partial D^{\ast}_1$ and $\partial E^{\ast}_1$ in the surface $\Sigma_0 = \Sigma / (D \cup E)$.
We divide into the following cases according to whether each of $\partial D^{\ast}_1$ and $\partial E^{\ast}_1$ is separating or not in $\Sigma_0$.
In each case, we get a desired reducing sphere for $\Sigma$, or derive a contradiction meaning that case cannot happen.

\medskip

\noindent Case $1$. Both $\partial D^{\ast}_1$ and $\partial E^{\ast}_1$ are non-separating in $\Sigma_0$.

\medskip

\noindent Case $1$-$1$. The curves $\partial D^{\ast}_1$ and $\partial E^{\ast}_1$ are not isotopic in $\Sigma_0$.
If we compress $\Sigma_0$ along $D^{\ast}_1 \cup E^{\ast}_1$, then the result is a $2$-sphere containing all the $8$ scars from $D$, $D^{\ast}_1$, $E$, $E^{\ast}_1$.
By Lemma \ref{lem2.4}, there exists a reducing sphere $P$ for $\Sigma$ such that $P$ cuts $\Sigma$ into two genus-$2$ summands and separates $D \cup D^{\ast}_1$ and $E \cup E^{\ast}_1$, and hence separates the pair $D - E$.

\medskip

\noindent Case $1$-$2$. The curves $\partial D^{\ast}_1$ and $\partial E^{\ast}_1$ are isotopic in $\Sigma_0$.
Let $A$ be the annulus in $\Sigma_0$ bounded by $\partial D^{\ast}_1 \cup \partial E^{\ast}_1$.
Then, a component of $\Sigma_0 / (D^{\ast}_1 \cup E^{\ast}_1)$ is a non-separating $2$-sphere in $S^3$, a contradiction.

\medskip

\noindent Case $2$. The curve  $\partial D^{\ast}_1$ is separating and the curve $\partial E^{\ast}_1$ is non-separating in $\Sigma_0$.
Let $S$ be the $2$-sphere component of $\Sigma_0 / (D^{\ast}_1 \cup E^{\ast}_1)$, which may contain $k$ scars from $E$, where $0 \le k \le 2$.

\medskip

\noindent Case $2$-$1$. $k = 0$.
In this case, $S$ contains no scar of $E$, two scars of $E^{\ast}_1$, and one scar of $D^{\ast}_1$.
Let $\gamma$ be a circle in $S$ separating the two scars of $E^{\ast}_1$ and the rest of the scars on $S$.
By Lemma \ref{lem2.4}, there exists a reducing sphere $P$ for $\Sigma$ such that $P \cap \Sigma = \gamma$.
The sphere $P$ cuts $\Sigma$ into a genus-$1$ summand and a genus-$3$ summand.
The genus-$1$ summand contains only the boundary circle $\partial E^{\ast}_1$, and so $\partial D$ and $\partial E$ are contained in the genus-$3$ summand.
Since the union of the genus-$3$ summand and the disk $P \cap V$ (or $P \cap W$) is a genus-$3$ Heegaard surface containing $\partial D$ and $\partial E$, at least one of $D$ and $E$ is a primitive disk by Lemma \ref{lem2.5}, a contradiction to our assumption.

\medskip

\noindent Case $2$-$2$. $k = 1$.
The $2$-sphere $S$ may contain $0$, $1$ or $2$ scars of $D$.
Suppose that $S$ contains no scar of $D$.
Then a component of $\Sigma / (D^{\ast}_1 \cup E)$ is a non-separating torus in $S^3$, a contradiction.
Suppose that $S$ contains one scar of $D$.
Then a component of $\Sigma / (D^{\ast}_1 \cup D \cup E)$ is a non-separating torus in $S^3$, a contradiction.
Suppose that $S$ contains two scars of $D$.
Then a component of $\Sigma / (D^{\ast}_1 \cup E)$ is a non-separating genus-$2$ surface in $S^3$, a contradiction.

\medskip

\noindent Case $2$-$3$. $k = 2$.
In this case, $S$ contains two scars of $E$, two scars of $E^{\ast}_1$, and one scar of $D^{\ast}_1$. 
Let $\gamma$ be a circle in $S$ separating the four scars of $E$ and $E^{\ast}_1$, and the rest of the scars on $S$.
By Lemma \ref{lem2.4}, there exists a reducing sphere $P$ for $\Sigma$ such that $P \cap \Sigma = \gamma$.
Then $P$ cuts $\Sigma$ into two genus-$2$ summands and separates $D \cup D^{\ast}_1$ and $E \cup E^{\ast}_1$, and hence separates the pair $D - E$.

\medskip

\noindent Case $3$. The curve $\partial E^{\ast}_1$ is separating and the curve $\partial D^{\ast}_1$ is non-separating in $\Sigma_0$.
Subcases in this case are largely symmetric to and similar to those of Case $2$. 
Considering the $2$-sphere component of $\Sigma_0 / (E^{\ast}_1 \cup D^{\ast}_1)$, which contains $k$ scars from $D$, where $0 \le k \le 2$, we have the similar conclusion to Case $2$.

\medskip

\noindent Case $4$. Both $\partial D^{\ast}_1$ and $\partial E^{\ast}_1$ are separating in $\Sigma_0$.
Since $\partial D^{\ast}_1$ and $\partial E^{\ast}_1$ are disjoint essential separating curves in a genus-$2$ surface, they are parallel.
Let $A$ be the annulus in $\Sigma_0$ bounded by $\partial D^{\ast}_1 \cup \partial E^{\ast}_1$.
The annulus $A$ may or may not contain scars of $D$ or $E$.
Suppose first that there is no scar of $D$ and $E$ in $A$. 
Choose in this case an essential circle $\gamma$ inside $A$.
Suppose next that there exist $k$ scars of $D$ or $E$ in $A$, where $1 \le k \le 4$.
Choose in this case a circle $\gamma$ inside $A$ separating the two unions; the union of $\partial D^{\ast}_1$ and the scars of $D$ in $A$ and the union of $\partial E^{\ast}_1$ and the scars of $E$ in $A$.
In any case, the circle $\gamma$ bounds a disk in $V$ that is disjoint from $D^{\ast}_1$ and $D$, and also bounds a disk in $W$ that is disjoint from $E^{\ast}_1$ and $E$.
Each component of $\Sigma - \gamma$ has genus at least $1$.
So the circle $\gamma$ is the intersection of $\Sigma$ and some reducing sphere $P$ for $\Sigma$ that is disjoint from $D^{\ast}_1 \cup E^{\ast}_1 \cup D \cup E$.

If $D$ and $E$ are lying in the opposite sides of $P$, then $P$ cuts $\Sigma$ into two genus-$2$ summands and separates the pair $D - E$ as desired.
If $D$ and $E$ are lying in the same side of $P$, then $P$ cuts $\Sigma$ into genus-$1$ and genus-$3$ summands, and $\partial D$ and $\partial E$ must lie in the genus-$3$ summand.
Since the union of the genus-$3$ summand and the disk $P \cap V$ (or $P \cap W$) is a genus-$3$ Heegaard surface containing $\partial D$ and $\partial E$, at least one of $D$ and $E$ is a primitive disk by Lemma \ref{lem2.5}, a contradiction to our assumption.
\end{proof}

\section{Proof of Theorem \ref{thm1.2}}\label{sec7}

Prior to proving Theorem \ref{thm1.2}, we first prove the following lemma, which is almost a direct consequence of Lemma \ref{lem1.1}.

\begin{lemma}\label{lem7.1} 
Let $(V, W; \Sigma)$ be the genus-$4$ Heegaard splitting of $S^3$.
Let $D - E$ be a non-separating weak reducing pair for $\Sigma$ such that $D \subset V$ and $E \subset W$.
If there exist non-separating compressing disks $D_0 \subset V$ and $E_0 \subset W$, respectively, such that 
\begin{itemize}
\item $D_0$ is not isotopic to $D$ in $V$ and $E_0$ is not isotopic to $E$ in $W$, and
\item each of $D_0$ and $E_0$ is disjoint from $D \cup E$, 
\end{itemize}
then there exists a reducing sphere $P$ that separates the pair $D - E$.
See Figure \ref{fig4}.
\end{lemma}

\begin{figure}[!hbt]
\includegraphics[width=3.2cm,clip]{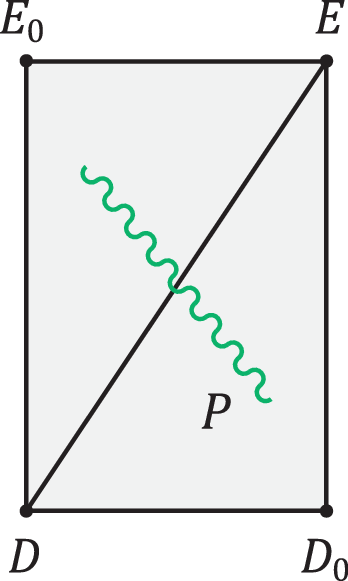}
\caption{A reducing sphere $P$ for $\Sigma$ separating $D$ and $E$.}
\label{fig4}
\end{figure}

We note that possibly the disks $D_0$ and $E_0$ in the statement might intersect each other.

\begin{proof}
If one of $D$ and $E$, say $D$, is primitive, then by Lemma \ref{lem2.1} we have a dual disk $\Delta \subset W$ of $D$ that is disjoint from $E$.
Then a reducing sphere $P$ for $\Sigma$ associated with $(D, \Delta)$ separates the pair $D - E$ as desired.
In this case, we do not need the disks $D_0$ and $E_0$.
From now on, we assume that both $D$ and $E$ are non-primitive.

\medskip

To have the conclusion, it is enough to show that both curves $\partial E_0$ and $\partial D_0$ are essential in the genus-$2$ surface $\Sigma_0 = \Sigma / (D \cup E)$.
If then, both disks $D_0$ and $E_0$ are compressing disks for $\Sigma_0$, which means $\Sigma_0$ is compressible to both sides.
Thus Lemma \ref{lem1.1} implies the existence of a reducing sphere $P$ separating the pair $D - E$.

\medskip

Suppose that $\partial E_0$ is inessential in $\Sigma_0$.
Let $\Delta$ be the disk in $\Sigma_0$ bounded by $\partial E_0$.
The disk $\Delta$ contains $k$ scars from $D$ or from $E$, where $1 \le k \le 4$.
We consider the cases in turn and obtain a contradiction in each case.

\medskip

\noindent Case $1$. $k = 1$. 
If $\Delta$ contains one scar of $D$, then the non-separating curve $\partial E_0$ in $\Sigma$ bounds a disk both in $V$ and $W$, a contradiction.
If $\Delta$ contains one scar of $E$, then it contradicts the assumption that $E_0$ is non-isotopic to $E$.

\medskip

\noindent Case $2$. $k = 2$.
If $\Delta$ contains two scars of $D$ or two scars of $E$, then it contradicts that $\partial E_0$ is non-separating in $\Sigma$.
If $\Delta$ contains one scar of $D$ and one scar of $E$, then a component of $\Sigma_0 / E_0$ is a non-separating $2$-sphere in $S^3$, a contradiction.

\medskip

\noindent Case $3$. $k = 3$.
Suppose that $\Delta$ contains two scars of $D$ and one scar of $E$.
Let $\gamma$ be a circle inside $\Delta$ separating the two scars of $D$ and the one scar of $E$.
By Lemma \ref{lem2.4}, there exists a reducing sphere $P$ such that $P \cap \Sigma = \gamma$ and $P$ separates $D$ and $E \cup E_0$.
The sphere $P$ cuts off a solid torus from $V$, and the disk $D$ lies in this solid torus.
It means that $D$ is a primitive disk, a contradiction to our assumption.
If $\Delta$ contains two scars of $E$ and one scar of $D$, then a component of $\Sigma_0 / E_0$ is a non-separating $2$-sphere in $S^3$, a contradiction.

\medskip

\noindent Case $4$. $k = 4$.
In this case, the disk $\Delta$ contains all the scars, two scars of $D$ and two scars of $E$. 
Then $\partial E_0$ is separating in $\Sigma$, a contradiction. 

\medskip

By a similar argument, we can see that $\partial D_0$ is also essential in $\Sigma_0$.
\end{proof}

\begin{proof}[Proof of Theorem \ref{thm1.2}]
The necessity is straightforward.
For sufficiency, let $P$ and $Q$ be any two reducing spheres for $\Sigma$.
We will show that $P \sim Q$ under the assumption of the theorem.
We first choose primitive disks $D$ and $D'$ in $V$ that are disjoint from $P$ and $Q$, respectively.
If $D$ and $D'$ are isotopic, then $P \sim Q$ by Lemma \ref{lem2.3}, so we may assume that $D$ and $D'$ are not isotopic.
Both $D$ and $D'$ are weak reducing disks for $\Sigma$.
Since the disk complex $\mathcal{D}(\Sigma)$ is simply connected \cite[Appendix A]{CKL24}, i.e. $\pi_1(\mathcal{D} (\Sigma)) = 1$, $\Sigma$ is not a critical surface.
By Lemma \ref{lem4.1}, there exists a weak reducing sequence $D = \Delta_1 - \Delta_2 - \cdots - \Delta_k = D'$ connecting $D$ to $D'$.
We take such a sequence of minimal length, that is, $\Delta_i$ and $\Delta_{i+2}$ are not isotopic for each $i \in \{1, 2, \ldots, k-2\}$.
By applying Lemma \ref{lem5.2} and Lemma \ref{lem5.3} repeatedly if necessary, we obtain a weak reducing sequence 
$$D = D_1 - E_2 - D_3 - \cdots - E_{2n} - D_{2n+1} = D'$$
consisting entirely of non-separating disks.
Again we take such a sequence of minimal length.
That is, we assume that $D_i$ and $D_{i+2}$ are not isotopic and $E_j$ and $E_{j+2}$ are not isotopic for any $i$ and $j$.

\medskip

\noindent {\it Claim}. There exists a reducing sphere $P_i$ for $\Sigma$ such that $P_i$ separates the weak reducing pair $D_i - E_{i+1}$ for each $i \in \{1, 3, \ldots, 2n-1\}$.
Similarly, there exists a reducing sphere $P_j$ for $\Sigma$ such that $P_j$ separates the weak reducing pair $E_j - D_{j+1}$ for each $j \in \{2, 4, \ldots, 2n\}$.

\begin{proof}[Proof of Claim]
It suffices to prove the first statement, namely the case of $D_i - E_{i+1}$ for each $i \in \{1, 3, \ldots, 2n-1\}$, since the second statement follows by symmetry.

First, consider $E_{i-1} - D_i - E_{i+1}$ in the sequence for each $i \in \{3, 5, \ldots, 2n-1\}$.
Suppose that $E_{i-1} \cap E_{i+1} \ne \emptyset$.
We assume that $E_{i-1}$ and $E_{i+1}$ intersect minimally. 
Let $\Delta$ be an outermost subdisk of $E_{i-1}$ cut off by an arc of $E_{i-1} \cap E_{i+1}$ that is outermost in $E_{i-1}$.
Since $E_{i+1}$ is a non-separating disk, at least one of the two disks, denoted by $E_0$, obtained by a surgery of $E_{i+1}$ along $\Delta$ is a non-separating disk.
The disk $E_0$ is disjoint from $E_{i+1}$, and is also disjoint from $D_i$ because both $E_{i-1}$ and $E_{i+1}$ are disjoint from $D_i$.
By the minimality of $| E_{i-1} \cap E_{i+1} |$, $E_0$ is non-isotopic to $E_{i+1}$.
If $E_{i-1} \cap E_{i+1} = \emptyset$, then let $E_0 = E_{i-1}$.
Then $E_0$ is a non-separating disk that is disjoint from $D_i$ and $E_{i+1}$, and non-isotopic to $E_{i+1}$.
Considering $D_i - E_{i+1} - D_{i+2}$ in the sequence for each $i \in \{3, 5, \ldots, 2n-1\}$ in a similar way, we also have a non-separating disk $D_0 \subset V$ that is disjoint from $D_i$ and $E_{i+1}$, and non-isotopic to $D_i$.

By applying Lemma \ref{lem7.1} to $\{ D_i, E_{i+1}, D_0, E_0 \}$, there exists a reducing sphere $P_i$ for $\Sigma$ such that $P_i$ separates $D_i - E_{i+1}$.

Next, consider $D_1 - E_2$ in the sequence.
By Lemma \ref{lem2.1}, there exists a dual disk $\Delta \subset W$ of $D_1$ that is disjoint from $E_2$.
Let $P_1$ be a reducing sphere for $\Sigma$ associated with $(D_1, \Delta)$.
Then $P_1$ separates $D_1 - E_2$.
\end{proof}

For each $k \in \{1, 3, \ldots, 2n-1\}$, let $P_k$ and $P_{k+1}$ be the reducing spheres from the claim, i.e. $P_k$ separates $D_k - E_{k+1}$, and $P_{k+1}$ separates $E_{k+1} - D_{k+2}$.
Both $P_k$ and $P_{k+1}$ are disjoint from the non-separating disk $E_{k+1}$.
By the assumption, we have $P_k \sim P_{k+1}$.

For each $k \in \{2, 4, \ldots, 2n-2\}$, let $P_k$ and $P_{k+1}$ be the reducing spheres from the claim, i.e. $P_k$ separates $E_k - D_{k+1}$, and $P_{k+1}$ separates $D_{k+1} - E_{k+2}$.
Both $P_k$ and $P_{k+1}$ are disjoint from the non-separating disk $D_{k+1}$.
By the assumption again, we have $P_k \sim P_{k+1}$.

Thus we see that $P_1 \sim P_{2n}$.
Furthermore, since both of the disks $D$($= D_1$) and $D'$($= D_{2n+1}$) are primitive, $P \sim P_1$ and $P_{2n} \sim Q$ by Lemma \ref{lem2.3}.
We conclude that $P \sim Q$.
\end{proof}

\bibliographystyle{amsplain}

\begin{thebibliography}{10}

\bibitem{Bachman1} D. Bachman,
Critical Heegaard surfaces,
Trans. Amer. Math. Soc. \textbf{354} (2002), no. 10, 4015--4042.

\bibitem{Bachman2} D. Bachman,
Connected sums of unstabilizd Heegaard splittings are unstabilized,
Geom. Topol. \textbf{12} (2008), no. 4, 2327--2378.

\bibitem{Bachman3} D. Bachman,
Topological index theory for surfaces in $3$-manifolds,
Geom. Topol. \textbf{14} (2010), no. 1, 585--609.

\bibitem{Casson-Gordon} A. J. Casson and C. McA. Gordon,
Reducing Heegaard splittings,
Topology Appl. \textbf{27} (1987), no. 3, 275--283.

\bibitem{CKL24} S. Cho, Y. Koda, and J. H. Lee,
The Powell Conjecture for the genus-three Heegaard splitting of the $3$-sphere,
Proc. Amer. Math. Soc. \textbf{154} (2026), no. 5, 2195--2208.

\bibitem{CKL25} S. Cho, Y. Koda, and J. H. Lee,
Reducing spheres and weak reducing pairs for Heegaard surfaces in the $3$-sphere,
arXiv:2509.24388.

\bibitem{Farb-Margalit} B. Farb and D. Margalit,
A primer on mapping class groups, 
Princeton Mathematical Series, 49, Princeton University Press, Princeton, NJ, 2012.

\bibitem{Gordon} C. McA. Gordon,
On primitive sets of loops in the boundary of a handlebody,
Topology Appl. \textbf{27} (1987), no. 3, 285--299.

\bibitem{Wal68} F. Waldhausen,
Heegaard-Zerlegungen der $3$-Sph\"{a}re, 
Topology \textbf{7} (1968), 195--203.

\bibitem{Zupan} A. Zupan,
The Powell conjecture and reducing sphere complexes,
J. Lond. Math. Soc. (2) \textbf{101} (2020), no. 1, 328--348.

\end{thebibliography}

\end{document}